\documentclass[12pt]{amsart}

\usepackage{amsmath,amsthm,amssymb,amsfonts}

\newtheorem{theorem}{Theorem}

\newcommand{\BH}{\mathrm{B}(\mathcal{H})}
\newcommand{\dfn}[1]{{\bf #1}}

\title[The Quantum Cross]{The spectral constant for the quantum cross and asymptotically sharp bounds for annuli}
\author{J. E. Pascoe}
\date{}
\begin{document}

\begin{abstract}
The quantum annulus of type $r$ is the class of invertible operators with singular values in $(1/r,r).$
Given an analytic function on the classical annulus of type $r,$ we may evaluate it on operators in the quantum annulus via the holomorphic functional calculus.
The spectral constant gives the maximum ratio betweeen the supremum over the norm of evalutions at operators in the quantum annulus to
the supremum over classical evaluations. We show that the limit of the spectral constant as $r$ goes to infinity is $2.$
Via the correspondence between annuli and hyperbolae, our study degenerates the problem to one on the quantum cross, pairs of contractions with product zero, where the spectral constant is exactly $2.$

The essential technique is to rationally dilate $Z$ to $\hat{Z}$ which has $U =(\hat{Z}+(\hat{Z}^{-1})^*)/(r+1/r)$ unitary
and estimate $Uf(\hat{Z})U^*$ directly.
\end{abstract}
\maketitle
\section{Introduction}
Let $r>1$
Let 
    $\mathbb{A}_r = \{z\in \mathbb{C}| |z|, |z^{-1}|<r\}$
denote the \dfn{annulus}.
We let $Q\mathbb{A}_r$ denote the Hilbert space operators $Z$ such that $\|Z\|,\|Z^{-1}\|<r,$ called the \dfn{quantum annulus}.

Let $f: \mathbb{A}_r \rightarrow \BH$ be analytic. We denote the infinity norm of $f$ by $\|f\|_{\mathbb{A}_r}.$
We denote the class of functions with finite norm by $H^\infty(\mathbb{A}_r).$
We define $f(Z)$ via evaluation of the Laurent series for $f.$
We denote the Agler norm of $f$ by $\|f\|_{Q\mathbb{A}_r}$ by the supremum of $\|f(Z)\|$ over $Z \in Q\mathbb{A}_r.$
We denote the class of functions with finite norm by $H^\infty(Q\mathbb{A}_r).$

The inclusion map from $\iota_r: H^\infty(\mathbb{A}_r) \rightarrow H^\infty(Q\mathbb{A}_r)$ is well-defined and bounded by a result of Shields \cite{shields}.
We call $K(r) = \|\iota_r\|$ the \dfn{spectral constant} of the quantum annulus of type $r.$
The best known lower bound is due to Tsikalas \cite{tsikalas} and upper bound due to Crouzeix and Greenbaum \cite{crouzeixgreenbaum}, specifically
    $2 \leq K(r)\leq 1+\sqrt{2}.$
Whether $K(r)$ is constant remains unclear. However it is known that \cite{pascnote},
    $$\liminf_{r\rightarrow \infty} K(r) = \inf K(r),$$
    $$\limsup_{r\rightarrow 1} K(r) = \sup K(r).$$

We establish that:
    $$K(r) \leq 2+O\left(\frac{1}{r^2}\right) \textrm{ as } r \rightarrow \infty$$
and thus the Tsikalas bound is asymptotically tight.

\subsection{Related domains}
Define the \dfn{conservative hyperbola} $\mathbb{H}_r$ to be the points $(z,w) \in \mathbb{D}^2$ such that $zw=1/r^2.$
Define the \dfn{quantum conservative hyperbola} $Q\mathbb{H}_r$ to be the pairs of Hilbert space operators $Z, W$ such that
$$ZW = 1/r^2, \|Z\|,\|W\|<1.$$ Note the map $Z \rightarrow (Z/r, Z^{-1}/r)$ from annuli to hyperbolas is bijective.

We call $\mathbb{H}_\infty= \{(z,w)\in \mathbb{D}^2| zw=0\}$ the \dfn{cross} and $Q\mathbb{H}_\infty$ the set of pairs $\|Z\|,\|W\|<1$ and $ZW=WZ=0,$ called the
\dfn{quantum cross}.

By the virtue that both conservative hyperbolas and the cross are subdomains of the bidisk, we have that the quantum norm of a function $f$ is exactly the minimum norm
$\hat{f}\in H^{\infty}(\mathbb{D}^2)$ whose restriction to the subdomain is $f$ via the Pick interpolation theory \cite{ampi}. That is,
our results give the following bound on the norm of a function off these varieties:
    $$\inf_{\hat{f}|_{H_r}=f \neq 0} \frac{\|\hat{f}\|_{\mathbb{D}^2}}{\|f\|_{\mathbb{H}_r}} \leq 2+O\left(\frac{1}{r^2}\right) \textrm{ as } r \rightarrow \infty.$$
Note the bound is sharp in the case of the cross. We denote the \dfn{spectral constant of the quantum cross} by $K(\infty).$

We will show that the spectral constant for the quantum cross is $2,$ the first $K(r)$ in our family under consideration known exactly.

Note taking $f(z,w)=z+w$ and 
$$Z=W=\begin{pmatrix}
0 & 1 \\ 0 & 0
\end{pmatrix}$$
gives that the lower bound is indeed $2.$

\section{Miniature dilations on quantum hyperbolae}

We now prove a perfect minidilation theorem for quantum hyperbolae. For large dilation theorems on quantum annuli, see \cite{pascmc}. For the minidilation theory on the noncommutative row ball, see \cite{apj}.
\begin{theorem}
    If $(Z, W) \in Q\mathbb{H}_r,$ there there exist $(\hat{Z},\hat{W})$ such that
    \begin{enumerate}
        \item $\|\hat{Z}\|,\|\hat{W}\|=1,$ 
        \item $\hat{Z}\hat{W}=\hat{W}\hat{Z}=1/r^2,$
        \item $\sigma(\hat{Z})\subseteq \sigma(Z)\cup\sigma(W^*), \sigma(\hat{W})\subseteq \sigma(W)\cup\sigma(Z^*),$
        \item $U=\frac{r}{r+1/r}\left(\hat{Z}+\hat{W}^*\right)$ is unitary,
        \item for any holomorphic $f:\mathbb{H}_r\rightarrow \mathbb{C},$
        $$\|f(Z,W)\|\leq \|f(\hat{Z},\hat{W})\|.$$
    \end{enumerate}
\end{theorem}
\begin{proof}
    Write a polar decomposition
        $Z= UA.$
    Note, $W=\hat{A}U^*$
    Take $H=U\sqrt{\left(\frac{r+1/r}{r}\right)^2-(A+\hat{A})^2}.$
    Let
        $$\hat{Z}= \begin{pmatrix}
            Z & H\\ 0 & W^*
        \end{pmatrix},
        \hat{W}= \begin{pmatrix}
            W & -H^* \\ 0 & Z^*
        \end{pmatrix}.$$
\end{proof}

\section{The estimate}
We now prove the main result
\begin{theorem}
$$K(r) \leq 2+O\left(\frac{1}{r^2}\right) \textrm{ as } r \rightarrow \infty$$
\end{theorem}
\begin{proof}
Write $f:\mathbb{H}_r\rightarrow \mathbb{C}$ analytic as
$f^+(z)z + a_0 +wf^-(w).$ Without loss of generality, we can take $U$ unitary as above.
Now,
\begin{align*}
\frac{r+1/r}{r}Uf(Z,W)U^*&=\frac{r+1/r}{r}U(f^+(Z)Z + a_0 +Wf^-(W))U^*\\& = U(f^+(Z)Z + a_0)Z^* + W^*(Wf^-(W)+a_0)U^* \\&+ Uf^+(Z)ZW+ ZWf^-(W)U^* \\ &
= U(f^+(Z)Z + a_0)Z^* + W^*(Wf^-(W)+a_0)U^* \\&+ \frac{1}{r^2}[Uf^+(Z)+f^-(W)U^*].
\end{align*}
(The small miracle that $Ua_0U^* = Ua_0Z^* + W^*a_0U^*,$ as $ZW=W^*Z^*$ handles the constant term. For the quantum cross, any quantity involving $r$ should be understood as the corresponding limit as $r \rightarrow \infty$.)
Note, by Cauchy's estimate,
$$\|f^+z + a_0\|_\mathbb{D},\|wf^- + a_0\|_\mathbb{D}\leq \frac{r^2}{r^2-1}\|f\|_{\mathbb{H}_r}$$
$$\|f^+\|_\mathbb{D},\|f^-\|_\mathbb{D}\leq \frac{2r^2-1}{r^2-1}\|f\|_{\mathbb{H}_r}.$$
So by von Neumann's inequality,
$$\|f(Z,W)\|\leq 2\frac{r^2}{r^2+1}\left[ \frac{r^2}{1-r^2} + \frac{2r^2-1}{r^4-r^2}\right]\|f\|_{\mathbb{H}_r}.$$
Thus,
$$\|f(Z,W)\|\leq 2\left[ 1+ \frac{2r^2}{r^4-1}\right]\|f\|_{\mathbb{H}_r}.$$
\end{proof}
\subsection{Some remarks}
Hartz noted to the author that Shields has somewhat sharper bounds for the norms of the decomposition into two functions on the disk,
but as we are merely concerned with asymptotics and the algebra is slightly more complicated.

\bibliographystyle{plain}
\bibliography{references}

\end{document}